\documentclass[11pt,oneside,reqno]{amsart}

\usepackage{graphicx}
\usepackage{tikz-cd}
\usepackage{color}
\usepackage[normalem]{ulem}
\usepackage{mathtools}
\usepackage{amsfonts}
\usepackage{amscd,amsmath}
\usepackage{amssymb}
\usepackage{setspace}
\usepackage{blindtext}
\usepackage{graphicx}
\usepackage{xfrac}
\usepackage{amsmath}

\usepackage{gastex}
\usepackage{longtable}
\usepackage{lscape}
\usepackage{tabularx}
\usepackage{multicol}
\usepackage{verbatim}
\usepackage{multirow}
\usepackage{epsfig,graphics,graphicx}
\usepackage{geometry}
\usepackage{amscd}
\usepackage{xcolor}
\usepackage{hyperref}
\usepackage[utf8]{inputenc}

\usepackage{amssymb,amsmath,amsthm, verbatim, calc}
\numberwithin{equation}{section}
\usepackage[english]{babel}

\definecolor{armygreen}{rgb}{0.29, 0.33, 0.13}
\definecolor{darkgreen}{rgb}{0.0, 0.2, 0.13}

\newtheorem{thm}{Theorem}[section]
\newtheorem{lem}[thm]{Lemma}
\newtheorem{cor}[thm]{Corollary}
\newtheorem{ex}[thm]{Example}

\newtheorem{Def}[thm]{Definition}
\newtheorem{prop}[thm]{Proposition}
\newtheorem{rem}[thm]{Remark}

\newcommand{\Rem}{\begin{rem} \rm}
\newcommand{\bdfn}{\begin{Def} \rm}
\newcommand{\edfn}{\end{Def}}

\newcommand{\ba}{\begin{array}}
\newcommand{\ea}{\end{array}}

\begin{document}
\title[Generalized idempotents on the space of analytic functions]{Generalized idempotents on the space of analytic functions with bounded derivatives}

\author{Himanshu Kumar}
\address{Department of Applied Sciences, Indian Institute of Information Technology Allahabad, Prayagraj-211015, U.P., India.}
\email{himanshumath1507@gmail.com}

\author{Abdullah Bin Abu Baker}
\address{Department of Applied Sciences, Indian Institute of Information Technology Allahabad, Prayagraj-211015, U.P., India.}
\email{abdullahmath@gmail.com}

\author{Fernanda Botelho}
\address{Department of Mathematical Sciences, The University of Memphis, Memphis, TN 38152, USA.}
\email{mbotelho@memphis.edu}

		

\thanks{The first-named author gratefully acknowledges the Ministry of Education, New Delhi (India), for financial support, and IIIT Allahabad, India, for providing the resources and infrastructure to carry out this research. The second-named author is partially supported by Anusandhan National Research Foundation (MATRICS) grant No. MTR/2022/000710.}

\date{\today}
		
\begin{abstract}
Let $X$ be a complex normed space. A map $P: X \rightarrow X$ is called idempotent if $P^2 = P$. A collection $\mathcal{C} = \{P_1, P_2\}$ of nonzero idempotent maps on $X$ is said to be a family of generalized bi-circular idempotents if there exist distinct unit modulus complex numbers $\lambda_1, \lambda_2$ such that $P_1 + P_2 = I$ (identity operator on $X$) and $\lambda_1P_1 + \lambda_2P_2$ is a surjective isometry on $X$. This generalizes the notion of generalized bi-circular projections on Banach spaces introduced by Fo\v{s}ner, Ili\v{s}evi\'{c} and Li \cite{MDC} to nonlinear maps. In this paper, we describe the structure of generalized bi-circular idempotents over the space of analytic functions on the open unit disk with bounded derivatives.	
\end{abstract}
		
\maketitle
		
\thispagestyle{empty}

\section{Introduction and Basic Results}

Let $X$ be a complex normed space. A map $P: X \rightarrow X$ is called idempotent if $P^2 = P$. A bounded (continuous) and idempotent linear operator on $X$ is called a projection. In this note, we extend a class of bi-contractive projections on Banach spaces, known as generalized bi-circular projections and introduced by Fo\v{s}ner, Ili\v{s}evi\'{c} and Li \cite{MDC}, to a nonlinear setting and study some of its properties. Let $I$ denote the identity operator on $X$, and let $\mathbb{T}$ denote the unit circle in the complex plane. A projection $P$ on $X$ is called a bi-circular projection if $P + \lambda (I-P)$ is a surjective isometry on $X$ for all $\lambda \in \mathbb{T}$. These projections were studied in different settings by Stach\'o and Zalar, see \cite{LB1, LB2}. Moreover, it was shown by Jamison that bi-circular projections are norm Hermitian \cite{JJ}. As a consequence, many results on bi-circular projections follow from previously known results on Hermitian operators. 

We recall that a map $T: X \rightarrow X$ is called an isometry if $\| Tx - Ty \| = \| x - y \|$ for all $x, y \in X$. Moreover, if $T$ is linear then $T$ is an isometry if and only if $||Tx|| = ||x||$ for all $x \in X$. Furthermore, if $T$ is an isometry, then $T^n$ is also an isometry, for $n \geq 2$.

The notion of bi-circular projection was generalized in \cite{MDC} by requiring that, for a projection $P$, there exists $\lambda \in \mathbb{T} \setminus \{1\}$ such that $P + \lambda (I-P)$ is a surjective isometry. Such a projection is called a generalized bi-circular projection. Over the last two decades, this class of projections has been extensively studied for several Banach spaces. We refer interested readers to the papers \cite{FJ, FJA, MDC, MH, DI1, DI2, P} and the references therein.

To extend the notion of generalized bi-circular projections to nonlinear maps, i.e., both $P$ (which is simply an idempotent map) and the isometry $T$, are not necessarily linear, we first observe that if $P$ is an idempotent map on $X$, then $I-P$ may not be idempotent. A simple example is to define $P: X \rightarrow X$ as $P(x) = \frac{x}{\|x\|}$ when $x \neq 0$ and $P(0) = 0$. We will revisit this example latter. Therefore, in our definition, we must assume that $ I-P$ is also idempotent. Thus, we have the following definition. 

\begin{Def} \label{GBCI} Let $X$ be a complex normed space. A collection $\mathcal{C} = \{P_1, P_2\}$ of nonzero idempotent maps on $X$ is said to be a family of {\bf \em generalized bi-circular idempotents} ($GBI$, for short) if $P_1 + P_2 =I$, and there exist distinct $\lambda_1$, $\lambda_2 \in \mathbb{T}$ such that $T = \lambda_1 P_1 + \lambda_2 P_2$ is a surjective isometry on $X$.

Each $P_i$, $i = 1,2$, is called a generalized bi-circular idempotent. We also say that $\mathcal{C} = \{P_1, P_2\}$ is a family of generalized bi-circular idempotents corresponding to the isometry $T$. We sometimes refer to $T$ as the isometry associated with the family $\mathcal{C}$. 
\end{Def}

\begin{rem}
If $\mathcal{C} = \{P_1, P_2\}$ is a family of generalized bi-circular idempotents, then $P_1P_2 = P_2P_1 = 0$. Indeed, $(P_1P_2)x = (I - P_2)(P_2(x)) = P_2(x) - P^2_2(x) = 0$, Similarly, $P_2P_1 = 0$.   
\end{rem}

Let $\mathbb{D}$ be the open unit disk in the complex plane. We consider the Hardy space $H^\infty(\mathbb{D})$, which is the commutative Banach algebra of all bounded analytic functions on $\mathbb{D}$, equipped with the supremum norm 
$$
\| f \|_\infty = \sup_{z \in \mathbb{D}} |f(z)|.
$$
We define $\mathcal{S}^\infty$ as the Banach space of all analytic functions $f$ on the open unit $\mathbb{D}$ whose derivative $f'$ belongs to $H^\infty(\mathbb{D})$ endowed with the norm 
$$
\|f\| = |f(0)| + \sup_{z \in \mathbb{D}} |f'(z)|.
$$
In this paper, we characterize $GBI$ and isometric reflections on the space $\mathcal{S}^\infty$. A map $T$ on a normed linear space $X$ is said to be of order $n$ if $T^n = I$ for some positive integer $n$ and $T^k \neq I$ for all positive integers $k<n$. If $n = 2$, $T$ is called a reflection. Moreover, an isometry $T$ of order $2$ is called an isometric reflection.

It is important to note that in \cite{FT}, Botelho and Miura initially introduced the notion of generalized bi-circular idempotents as follows. An idempotent map $P$ is called a generalized bi-circular idempotent if there exists a unimodular complex number $\lambda \neq 1$ such that $P+ \lambda (I - P)$ is a surjective isometry. Moreover, in \cite{FT, FT2}, the authors studied generalized bi-circular idempotents on the space of continuously differentiable complex-valued functions on $[0,1]$. 

In our definition of generalized bi-circular idempotents (Definition \ref{GBCI}), we have emphasized that $I - P$ is also an idempotent map. Unlike projections ($P$ is a projection if and only if $ I - P$ is a projection), this is a crucial assumption for the case of idempotent maps. Our next definition and lemma address this issue. 

\begin{Def}
A map $P$ on a normed space $X$ is said to be {\bf \em bi-potent} if both $P$ and $I-P$ are idempotent maps.
\end{Def}

\begin{lem}
Let $P$ be a map on a normed space $X$. Then $P$ is idempotent if and only if $(I - P)P = 0$. As a consequence, $P$ is bi-potent if and only if $(I - P)P = 0$ and $P(I-P) = 0$.
\end{lem}

\begin{proof}
If $P$ is an idempotent map, then $(I - P)P = P - P^2 = P - P = 0$. Reciprocally, if $(I - P)P = 0$, then $P- P^2=0$, that is $P=P^2$.
	
For the second part, based on what we have already proven, we have that $P$ and $I-P$ are idempotents if and only if $(I-P)P = 0$ and $(I-(I-P))(I-P)=0$, or equivalently, $(I-P)P = 0$ and $P(I-P)=0$. 
\end{proof}	

An example of an idempotent map that is not bi-potent is given below.

\begin{ex}
Let $P: X \rightarrow X$ be defined as $P(x) = \frac{x}{\|x\|}$ when $x \neq 0$ and $P(0) = 0$. Then $P$ is idempotent, but $I-P$ is not. To verify the latter claim, consider a nonzero $x \in X$ such that $\|x\| \neq 1$. Then $(I - P)x = x - \frac{x}{\|x \|} \neq 0$. It follows that $P(I-P)x \neq 0$. Therefore, $I-P$ is not an idempotent map. 
\end{ex}

\begin{lem} \label{lemma}
Let $\mathcal{C} = \{P_1, P_2 \}$ be a family of generalized bi-circular idempotents such that $T = \lambda_1 P_1 +\lambda_2 P_2$. Then $TP_1 = \lambda_1 P_1$ and $TP_2 = \lambda_2 P_2$.
\end{lem}

\begin{proof}
Multiplying both sides by $P_1$ and $P_2$ and using the orthogonality of $P_1$ and $P_2$, we obtain the desired result. 
\end{proof} 

Consider a family $\mathcal{C} = \{P_1, P_2 \}$ of $GBI$ corresponding to the isometry $T$. The Mazur-Ulam theorem \cite{MU} states that some translation of $T$ is real linear isometry. In other words, $T = T(0) + S$, where $S$ is real linear isometry on $X$. Our next proposition establishes that $T$ is real linear.

\begin{prop} \label{T0=0}
Let $\mathcal{C} = \{P_1, P_2 \}$ be a family of generalized bi-circular idempotents corresponding to the isometry $T$. Then $T(0) = 0$. Moreover, $T$ is real linear.
\end{prop}

\begin{proof}
Let $T = \lambda_1 P_1 + \lambda_2 P_2$, where $\lambda_1, \lambda_2$ are distinct complex numbers of modulus one. By the Mazur-Ulam theorem, we can write $T = T(0) + S$, where $S$ is real linear surjective isometry. By Lemma \ref{lemma}, we have $TP_1 = \lambda_1 P_1$ and $TP_2 = \lambda_2 P_2$. Consequently, $T(0) +  SP_1 = \lambda_1 P_1$ and $T(0) +  SP_2 = \lambda_2 P_2$. Adding these two equations, we get $2 T(0) +  S P_1 + S P_2 =  \lambda_1P_1 + \lambda_2 P_2$. Since $S$ is real linear and $P_1 + P_2 = I$, it follows that $2T(0) + S = T(0) + S$. Hence, $T(0) = 0$, and thus, $T$ is real linear.
\end{proof}

\begin{cor} \label{gbireal}
Let $\mathcal{C} = \{P_1, P_2\}$ be a family of generalized bi-circular idempotents on a normed space $X$. Then $P_1$ and $P_2$ are real linear.
\end{cor}

\begin{proof}
Let $T= \lambda_1P_1+ \lambda_2P_2$, for some distinct $\lambda_1, \lambda_2 \in \mathbb{T}$. Since $P_1 + P_2 = I$, we conclude that
\begin{equation} \label{gbiidentity}
P_1 = \frac{T- \lambda_2 I}{\lambda_1 -\lambda_2} \ \text{and} \ P_2 = \frac{T- \lambda_1 I}{\lambda_2 -\lambda_1}.
\end{equation}
Since $T$ is real linear, $P_1$ and $ P_2$ automatically become real linear.
\end{proof}

\begin{rem}
The above Proposition and Corollary imply only that the maps under consideration are real linear, not complex linear. So, essentially, they are nonlinear maps. 
\end{rem}

\begin{ex}
Define the real linear idempotents $P_1, P_2 : \mathbb{C} \to \mathbb{C}$, by
$$
P_1(z=x +iy) = x,\ \text{ and } \  P_2(z=x +iy) = iy \ \text{ for all } \ x,y \in \mathbb{R} .
$$
Obviously, $P_1 +P_2 =I$ and $P_1P_2=0 =P_2P_1$. For any $x,y \in \mathbb{R}$, let 
$$
T(x+iy) = \lambda_1P_1 + \lambda_2 P_2(x+iy) =  \lambda_1 x + i \lambda_2 y,
$$
where $\lambda_1 \neq \lambda_2 \in \mathbb{T}$. It follows that
\begin{equation*}
|T(z)|^2 = |\lambda_1|^2 x^2 + |\lambda_2|^2 y^2 + xy\big( -i \lambda_1\overline{\lambda_2} + i \overline{\lambda_1}\lambda_2\big) = x^2 + y^2 - 2xy\,\Im \big(\overline{\lambda_1} \lambda_2\big),
\end{equation*}
where $\Im\big(\overline{\lambda_1} \lambda_2\big)$ denotes the imaginary part of $\overline{\lambda_1} \lambda_2$. Now, $T$ is an isometry if and only if $|T(z)|^2 = |z|^2 = x^2 + y^2$, for all $x,y \in \mathbb{R}$. Hence, $\Im\big(\overline{\lambda_1} \lambda_2\big) = 0$. Since $\overline{\lambda_1}\lambda_2 \in \mathbb{T}$ and $\lambda_1 \neq \lambda_2$, this concludes that $\lambda_1=-\lambda_2$. So, $T(z) = -\lambda_2 \overline{z} $. Clearly, $T$ is a surjective isometry on $\mathbb{C}$. Therefore, the collection $\{P_1, P_2\}$ is a family of generalized bi-circular idempotents for $\lambda_1 = -\lambda_2$.
\end{ex}

The structure of surjective isometries on the space $\mathcal{S}^\infty$ was characterized by Miura and Niwa in \cite{MN1}, and is stated in the following theorem.

\begin{thm}\cite[Theorem 2]{MN1} \label{iso}
If $T: \mathcal{S}^\infty \rightarrow \mathcal{S}^\infty$ is a surjective, not necessarily linear, isometry with respect to the norm $\|f\| = |f(0)| + \sup_{z \in \mathbb{D}} |f'(z)|$ for $f \in \mathcal{S}^\infty$,  then there exist constants $c_0, c_1, \mu \in \mathbb{T}$ and $a \in \mathbb{D}$ such that for all $f \in \mathcal{S}^\infty$, and $z \in \mathbb{D}$ 
\begin{equation*}
T(f)(z)= T(0)(z)+ c_0f(0) + \int_{[0,z]}c_1f'(\psi(\xi))d\xi, \tag*{(Form I)} \text{ or } 
\end{equation*}
\begin{equation*}
T(f)(z)= T(0)(z)+ c_0 \overline{f(0)} + \int_{[0,z]} c_1f'(\psi(\xi))d\xi, \tag*{(Form II)} \text{ or } 
\end{equation*}
\begin{equation*}
T(f)(z)= T(0)(z)+ c_0f(0) + \int_{[0,z]} c_1 \overline{f'(\psi(\overline{\xi}))}d\xi, \tag*{(Form III)} \text{ or } 
\end{equation*}
\begin{equation*}
T(f)(z)= T(0)(z)+ c_0\overline{f(0)} + \int_{[0,z]} c_1\overline{f'(\psi(\overline{\xi}))} d\xi, \tag*{(Form IV)} 
\end{equation*}
where $\psi(z) = \mu \frac{z-a}{\bar{a}z-1}$ for all $z \in \mathbb{D}$.  
\end{thm}

\begin{ex}
Consider the map $P_1 : \mathcal{S}^{\infty} \to \mathcal{S}^{\infty}$ defined by
\begin{equation*}
P_1(f)(z)= -1 + \frac{f(0)}{2} -\frac{\overline{f(0)}}{2},  
\end{equation*}
for all $f\in\mathcal{S}^{\infty}$ and $z\in\mathbb{D}$. Then, $P_1$ is idempotent, whereas $P_2= I - P_1$ is not idempotent. Nevertheless, the map $T = \lambda_1P_1 + \lambda_2P_2$, with $\lambda_1 = 1$ and $\lambda_2 = -1$, is a surjective isometry on $\mathcal{S}^{\infty}$.
\end{ex}

\section{Isometric Reflections on $\mathcal{S}^\infty$}

In this section, we find the structure of the isometric reflections on $\mathcal{S}^\infty$. We first consider an example which shows that the isometry associated with a family of generalized bi-circular idempotents is not always a reflection. 

\begin{ex}
Define $P_1, P_2 : \mathcal{S}^{\infty} \to \mathcal{S}^{\infty}$ by
\begin{align*}
P_1(f)(z) = \frac{f(0) - i \overline{f(0)}}{2} \ \ \text{ and } \  \  P_2(f)(z) = f(z) + \frac{ i \overline{f(0)} - f(0)}{2}
\end{align*}
for all $f\in\mathcal{S}^{\infty}$ and $z\in\mathbb{D}$. Then, $P_1$ and $P_2$ are idempotent maps. Moreover, the map $T= iP_1 -iP_2$ is a surjective isometry of Form $(II)$, but it is not a reflection.
\end{ex}

However, we will see in the next section that the isometry associated with a family of generalized bi-circular idempotents is close to being an isometric reflection. 

\begin{thm} \label{ir1}
An isometry $T: \mathcal{S}^\infty \to \mathcal{S}^\infty$ of Form $(I)$ is a reflection if and only if the following conditions hold:
\begin{enumerate}
\item[(i)] $c_0, c_1 \in \{\pm 1\}$,
\item[(ii)] $\psi: \mathbb{D} \to \mathbb{D}$ satisfies $\psi^2(z) = z$ for all $z \in \mathbb{D}$, and
\item[(iii)] $T(0)(z) + c_0 T(0)(0) + \int_{[0,z]} c_1 T(0)'(\psi(\xi)) \, d\xi = 0   \text{ for all } z \in\mathbb{D}$.
\end{enumerate}
\end{thm}
\begin{proof}
$(\Longrightarrow)$ An isometry on $\mathcal{S}^\infty$ of Form $(I)$ is given by
\begin{equation*}
T(f)(z)= T(0)(z)+ c_0f(0) + \int_{[0,z]} c_1f'(\psi(\xi))d\xi,\ \forall \ f \in \mathcal{S}^\infty \text{ and } z \in \mathbb{D}.
\end{equation*}
This implies that 	
\begin{align*}
T(f)(0) = T(0)(0) + c_0f(0) \text{ and } T(f)'(z)= T(0)'(z)+ c_1 f'(\psi(z)).	    
\end{align*}
Since $T$ is an isometric reflection, we have $T^2(f)(z) = f(z)$ for all $f \in \mathcal{S}^\infty$ and $z \in \mathbb{D}$. It follows that 
\begin{equation*}
T(0)(z)+ c_0 T(f)(0) + \int_{[0,z]}c_1 T(f)'(\psi(\xi))d\xi = f(z).
\end{equation*}
Substituting the values of $T(f)(0)$ and $T(f)'(\psi(\xi))$ in the above equation, we obtain
\begin{equation} \label{mainref1}
T(0)(z)+ c_0 \left[ T(0)(0)+c_0f(0) \right]  +  c_1 \int_{[0,z]} \left[  T(0)'(\psi(\xi))+ c_1 f'(\psi^2(\xi)) \right] d\xi = f(z). 
\end{equation}
Choosing $f= 0$, we get
\begin{equation*} \label{nlt1}
T(0)(z) +c_0 T(0)(0)+ c_1 \int_{[0,z]} T(0)'(\psi(\xi)) d\xi =0.
\end{equation*}
If we put $z = 0$, in the above equation, we get $(c_0 + 1) T(0)(0) = 0$. Moreover, Equation \eqref{mainref1} takes the form	
\begin{equation} \label{maint1}
c_0^2f(0) + c_1^2\int_{[0,z]} f'(\psi^2(\xi))d\xi = f(z).
\end{equation}	
Now, by choosing $f = 1$ we conclude that $c_0^2 =1$. Thus,  $c_0 =\pm 1$. 
	
Differentiating Equation \eqref{maint1}, we get $c_1^2 f'(\psi^2(z))= f'(z)$. This implies that $c_1 =\pm 1$ and $\psi^2(z)=z$. 

$(\Longleftarrow)$ Suppose that conditions (i)-(iii) hold. We now prove that $T^2(f)=f$ for all $f \in \mathcal{S}^\infty$.
For this,
\begin{align*} 
T^2(f)(z) &= T(0)(z)+ c_0 \left[ T(0)(0)+c_0f(0) \right]  +  c_1 \int_{[0,z]} \left[  T(0)'(\psi(\xi))+ c_1 f'(\psi^2(\xi)) \right] d\xi \\
&= \Big[ T(0)(z) + c_0 T(0)(0)  +  c_1 \int_{[0,z]}   T(0)'(\psi(\xi)) d\xi  \Big] + c_0^2f(0) + c_1^2  \int_{[0,z]}  f'(\psi^2(\xi))  d\xi. 
\end{align*}
By condition (iii), it follows that
\begin{equation*}
T^2(f)(z) = c_0^2f(0) + c_1^2  \int_{[0,z]}  f'(\psi^2(\xi))  d\xi. 
\end{equation*}
Using conditions (i) and (ii), we obtain
\begin{align*} 
T^2(f)(z)&= f(0) + \int_{[0,z]} f'(\xi)  d\xi \\
 &= f(0) + f(z) - f(0) = f(z).
\end{align*}
Therefore, $T^2(f)=f$ for all $f\in\mathcal{S}^\infty$, and hence $T$ is an isometric reflection on $\mathcal{S}^\infty$. This completes the proof.
\end{proof}

\begin{thm} \label{ir2}
An isometry $T: \mathcal{S}^\infty \to \mathcal{S}^\infty$ of Form $(II)$ is a reflection if and only if the following conditions hold:
\begin{enumerate}
\item[(i)] $c_1 \in \{\pm 1\}$,
\item[(ii)] $\psi: \mathbb{D} \to \mathbb{D}$ satisfies $\psi^2(z) = z$ for all $z \in \mathbb{D}$, and
\item[(iii)] $T(0)(z) + c_0 \overline{T(0)(0)} + c_1 \int_{[0,z]} T(0)'(\psi(\xi)) d\xi =0   \text{ for all } z \in\mathbb{D}$.
\end{enumerate}
\end{thm}

\begin{proof} 
$(\Longrightarrow)$
An isometry of Form $(II)$ is given by 
\begin{equation*}\label{t2}
T(f)(z)= T(0)(z)+ c_0\overline{f(0)} + \int_{[0,z]}c_1f'(\psi(\xi))d\xi,\ \forall \ f \in \mathcal{S}^\infty \text{ and } z \in \mathbb{D}.
\end{equation*}
Thus, 
$$T(f)(0) = T(0)(0) + c_0 \overline{f(0)} \text{ and } T(f)'(z)= T(0)'(z)+ c_1 f'(\psi(z)).$$
	
Since $T$ is an isometric reflection, $T^2(f)(z) = f(z)$ for all $f \in \mathcal{S}^\infty$ and $z \in \mathbb{D}$. This implies that 
\begin{equation*} \label{ct2}
T(0)(z)+ c_0\overline{T(f)(0)}  + \int_{[0,z]}c_1 T(f)'(\psi(\xi))d\xi =f(z),
\end{equation*}
or
\begin{equation} \label{mainref2}
T(0)(z)+ c_0\overline{T(0)(0)} + f(0) +  c_1 \int_{[0,z]} \left[T(0)'(\psi(\xi))+ c_1 f'(\psi^2(\xi))\right]d\xi = f(z).
\end{equation}
Choosing $f=0$, we get
\begin{equation*} \label{nlt2}
T(0)(z) + c_0\overline{T(0)(0)} + c_1 \int_{[0,z]} T(0)'(\psi(\xi)) d\xi =0.
\end{equation*}
So, Equation \eqref{mainref2} takes the form	
\begin{equation*} \label{maint2}
f(0) + c_1^2\int_{[0,z]} f'(\psi^2(\xi))d\xi = f(z).
\end{equation*}	
After differentiating this equation, we get $c_1^2 f'(\psi^2(z))= f'(z)$. It follows that $c_1 =\pm 1$ and $\psi^2(z)=z$.

$(\Longleftarrow)$ Suppose that conditions (i)-(iii) are satisfied. We show that $T^2(f) = f$ for all $f \in \mathcal{S}^\infty$. 
Indeed,
\begin{align*} 
T^2(f)(z) &= T(0)(z)+ c_0 \overline{T(0)(0) } + f(0) +  c_1 \int_{[0,z]} \left[T(0)'(\psi(\xi))+ c_1 f'(\psi^2(\xi))\right]d\xi.
\end{align*}
Now, by using (i), (ii), and (iii), we get
\begin{equation*} 
T^2(f)(z) = f(0) + \int_{[0,z]} f'(\xi)  d\xi = f(0) + f(z) - f(0) = f(z).
\end{equation*}
Consequently, $T^2(f)=f$ for all $f\in\mathcal{S}^\infty$, hence $T$ is an isometric reflection on $\mathcal{S}^\infty$. This completes the proof.
\end{proof}

For isometries of Forms $(I)$ and $(II)$, the reflection conditions depend on $\psi^2$, whereas for Forms $(III)$ and $(IV)$, they involve complex conjugation. The following results give the reflection conditions for isometries of Forms $(III)$ and $(IV)$.

\begin{thm} \label{ir3}
An isometry $T: \mathcal{S}^\infty \to \mathcal{S}^\infty$ of Form $(III)$ is a reflection if and only if the following conditions hold:
\begin{enumerate}
\item[(i)] $c_0 \in \{\pm 1\}$,
\item[(ii)] $\psi: \mathbb{D} \to \mathbb{D}$ satisfies $\psi(\overline{\psi(\overline{z})})= z$ for all $z \in \mathbb{D}$, and
\item[(iii)] $T(0)(z) + c_0\overline{T(0)(0)} +  \int_{[0,z]} c_1 \overline{ T(0)'(\psi(\overline{\xi}))} d\xi =0   \text{ for all } z \in\mathbb{D}$.
\end{enumerate}
\end{thm}

\begin{proof}

$(\Longrightarrow)$ The form of $T$ is given by 
\begin{equation*} \label{t3}
T(f)(z)= T(0)(z)+  c_0f(0) + \int_{[0,z]} c_1\overline{f'(\psi(\overline{\xi}))}d\xi \ \ \forall \ f \in \mathcal{S}^\infty \text{ and } z \in \mathbb{D}.
\end{equation*}

It follows that $T(f)(0) = T(0)(0) +  c_0f(0)$. Furthermore, from the Taylor series expansion of $f'$, it can be check that the function $z \mapsto \overline{f'(\psi(\overline{z}))}$ is holomorphic on $\mathbb{D}$. Thus, 
\begin{align*}
T(f)'(z)= T(0)'(z)+ c_1 \overline{f'(\psi(\overline{z}))}.
\end{align*}

Since $T^2(f)(z) = f(z)$ for all $f \in \mathcal{S}^\infty$ and $z \in \mathbb{D}$, we have 
\begin{equation*} \label{ct3}
T(0)(z)+ c_0T(f)(0)  + \int_{[0,z]} c_1 \overline{T(f)'(\psi(\overline{\xi}))}d\xi =f(z).
\end{equation*}
This implies that 
\begin{equation} \label{mainref3}
T(0)(z)+ c_0T(0)(0) + c_0^2 f(0) +  \int_{[0,z]} \left[c_1 \overline{ T(0)'(\psi( \overline{\xi}))} + f'(\psi (\overline{\psi(\overline{\xi})})) \right] d\xi = f(z).
\end{equation}
Selecting $f=0$, 
\begin{equation*} \label{nlt3}
T(0)(z) + c_0T(0)(0) +  \int_{[0,z]} c_1 \overline{ T(0)'(\psi(\overline{\xi}))} d\xi =0.
\end{equation*}
So, Equation \eqref{mainref3} becomes 
\begin{equation*} \label{maint3}
c_0^2 f(0) + \int_{[0,z]}f'(\psi (\overline{\psi(\overline{\xi}))})d\xi = f(z).
\end{equation*}	
Hence, $c_0 =\pm 1$ and $\psi(\overline{\psi(\overline{z})})= z$. Since $\psi(z) = \mu \frac{z-a}{\bar{a}z-1}$, putting $z = \overline{a}$ in the identity $\psi(\overline{\psi(\overline{z})})= z$, we obtain $\psi (0) = \overline{a}$. This implies that $\mu a = \overline{a}$. 

$(\Longleftarrow)$ Assume that conditions (i)-(iii) are satisfied. We claim that $T^2(f)=f$ for all $f \in \mathcal{S}^\infty$. 
To this end, consider
\begin{align*} 
T^2(f)(z) &= T(0)(z)+ c_0T(0)(0) + c_0^2 f(0) +  \int_{[0,z]} \left[c_1\overline{ T(0)'(\psi( \overline{\xi}))} + f'(\psi (\overline{\psi(\overline{\xi})})) \right] d\xi.
\end{align*}
In view of conditions (i)-(iii), we deduce that
\begin{equation*} 
T^2(f)(z)= f(0) + \int_{[0,z]} f'(\xi)  d\xi = f(0) + f(z) - f(0) = f(z).
\end{equation*}
Therefore, $T$ is an isometric reflection on $\mathcal{S}^\infty$. This completes the proof.
\end{proof}

\begin{thm}\label{ir4}
An isometry $T: \mathcal{S}^\infty \to \mathcal{S}^\infty$ of Form $(IV)$ is a reflection if and only if the following conditions hold:
\begin{enumerate} 
\item[(i)] $\psi: \mathbb{D} \to \mathbb{D}$ satisfies $\psi(\overline{\psi(\overline{z})})= z$ for all $z \in \mathbb{D}$, and
\item[(ii)] $T(0)(z) + c_0\overline{T(0)(0)} +  \int_{[0,z]} c_1\overline{ T(0)'(\psi(\overline{\xi}))} d\xi =0   \text{ for all } z \in\mathbb{D}$.
\end{enumerate}
\end{thm}

\begin{proof} 

$(\Longrightarrow)$ The isometry $T$ is given by 
\begin{equation*} \label{t4}
T(f)(z)= T(0)(z)+  c_0\overline{f(0)} + \int_{[0,z]} c_1\overline{f'(\psi(\overline{\xi}))}d\xi,\ \forall \ f \in \mathcal{S}^\infty \text{ and } z \in \mathbb{D}.
\end{equation*}
It is immediate that 
\begin{align*}
T(f)(0) = T(0)(0) +  c_0\overline{f(0)} \text{ and } T(f)'(z)= T(0)'(z)+ c_1 \overline{f'(\psi(\overline{z}))}.
\end{align*}
Moreover, $T^2(f)(z) = f(z)$ implies that 

\begin{equation*} \label{ct4}
T(0)(z)+ c_0\overline{T(f)(0)}   + \int_{[0,z]} c_1\overline{ T(f)'(\psi(\overline{\xi}))}d\xi = f(z).
\end{equation*}
	
Therefore, 
	
\begin{equation} \label{mainref4}
T(0)(z) + c_0\overline{T(0)(0)} +  f(0) +  \int_{[0,z]} \left[ c_1 \overline{ T(0)'(\psi( \overline{\xi}))} + f'(\psi \overline{\psi(\overline{\xi})})) \right] d\xi = f(z). 
\end{equation}
	
Taking $f=0$, we obtain
	
\begin{equation*} \label{nlt4}
T(0)(z) + c_0 \overline{T(0)(0)} +  \int_{[0,z]} c_1  \overline{T(0)'(\psi(\overline{\xi}))} d\xi = 0.
\end{equation*}
	
Thus, Equation \eqref{mainref4} takes the form 

\begin{equation*} \label{maint4}
f(0) + \int_{[0,z]}f'(\psi (\overline{\psi(\overline{\xi}))})d\xi = f(z).
\end{equation*}	
	
Therefore, $\psi(\overline{\psi(\overline{z})})= z$. This implies that $\mu a = \overline{a}$. 

$(\Longleftarrow)$ Let conditions (i) and (ii) be satisfied. We proceed to verify that $T^2(f)=f$ for all $f \in \mathcal{S}^\infty$. Observe that
\begin{align*} 
T^2(f)(z) &= T(0)(z)+ c_0 \overline{T(0)(0)} +  f(0) +  \int_{[0,z]} \left[ c_1 \overline{ T(0)'(\psi( \overline{\xi}))} + f'(\psi (\overline{\psi(\overline{\xi})})) \right] d\xi.
\end{align*}
Finally, using conditions (i) and (ii), we conclude that
\begin{equation*} 
T^2(f)(z)= f(0) + \int_{[0,z]} f'(\xi)  d\xi = f(0) + f(z) - f(0) = f(z).
\end{equation*}
Consequently, $T$ is an isometric reflection on $\mathcal{S}^\infty$, which completes the proof.

\end{proof}

\section{Structure of Generalized bi-circular idempotents on $\mathcal{S}^\infty$}

In this section, we characterize generalized bi-circular idempotents on $\mathcal{S}^\infty$.

\begin{thm} \label{formI}
If the collection $\mathcal{C} = \{P_1, P_2\}$ is a family of generalized bi-circular idempotents on $\mathcal{S}^\infty$ corresponding to an isometry $T$ of the Form $(I)$, then one of the following holds:
\begin{enumerate}
\item $\mathcal{C}$ is a family of generalized bi-circular projections such that each $P_i$ is the average of the identity operator and an isometric reflection. Moreover, we have $\lambda_1 + \lambda_2 = 0$, $c_0, c_1 \in \{\pm \lambda_1\}$ and $\psi^2(z) = z$ for all $z \in \mathbb{D}$.

\item $\mathcal{C}$ is a family of bi-circular projections such that each $P_i$ is the average of the identity operator and an isometric reflection. In this case, $c_0, c_1 \in \{\lambda_1,\lambda_2\}$, $c_1 \neq c_2$ and $\psi(z) = z$ for all $z \in \mathbb{D}$.
\end{enumerate} 
\end{thm}

\begin{proof}

Let $\mathcal{C} = \{P_1, P_2\}$ be a family of generalized bi-circular idempotents on $\mathcal{S}^\infty$ corresponding to an isometry $T$ of Form $(I)$. Then $T = \lambda_1P_1 + \lambda_2 P_2$, where $\lambda_1, \lambda_2 \in \mathbb{T}$. Proposition \ref{T0=0} implies that $T(0) = 0$. Thus, $T$ takes the form
\begin{equation*} \label{linearisometry}
T(f)(z)= c_0f(0) + \int_{[0,z]}c_1f'(\psi(\xi))d\xi, \ \forall \ f \in \mathcal{S}^\infty \text{ and } z \in \mathbb{D}. 
\end{equation*}
We observe that $T$ is a linear surjective isometry on $\mathcal{S}^\infty$. It follows from Equation \eqref{gbiidentity} that $P_1$, $P_2$ are projections and $(T - \lambda_1 I)(T - \lambda_2 I) = (\lambda_2- \lambda_1)P_2 (\lambda_1- \lambda_2)P_1 =0$. Thus,
\begin{equation} \label{identityl}
T^2(f)(z)-(\lambda_1+\lambda_2)T(f)(z)+\lambda_1\lambda_2 f(z)=0, \ \forall \ f \in \mathcal{S}^\infty  \text{ and } z \in \mathbb{D}.
\end{equation}
Further, 
\begin{equation}  \label{p}
P_i = \frac{T - \lambda_j I}{\lambda_i - \lambda_j},\ i,j = 1,2,\ i \neq j.
\end{equation}
Now, using the form of $T$ in Equation \eqref{identityl}, we obtain
\begin{equation} \label{main1}
\left[c_0^2 f(0) + c_1^2\int_{[0,z]}f'(\psi^2(\xi))d\xi\right] -(\lambda_1+\lambda_2)\left[ c_0f(0) + c_1 \int_{[0,z]}f'(\psi(\xi))d\xi\right] + \lambda_1\lambda_2 f(z) =0.
\end{equation}
Choosing $f = 1$ in the above equation, we get
\begin{equation*}
c_0^2 -(\lambda_1+\lambda_2)c_0 + \lambda_1\lambda_2 =0 \implies c_0 =\lambda_1,\lambda_2.
\end{equation*}
Again, choosing $f= id$, we have
\begin{equation*}
c_1^2 z -(\lambda_1+\lambda_2)c_1z +\lambda_1\lambda_2 z  = 0 \implies  c_1 =\lambda_1,\lambda_2.
\end{equation*}
Differentiating Equation \eqref{main1}, 
\begin{equation} \label{main1der}
c_1^2 f'(\psi^2(z))- (\lambda_1+\lambda_2) c_1 f'(\psi(z)) +\lambda_1\lambda_2 f'(z)=0.
\end{equation}
We claim that $\psi^2(z) =z$ for all $z \in \mathbb{D}$. Suppose that $\psi^2(z) \neq z$ for some $z \in \mathbb{D}$. Then $\psi(z) \neq z$ as well. Choosing a function $f \in \mathcal{S}^\infty$ such that $f'(z) = 1$ and $f'(\psi (z)) = f'(\psi^2(z)) = 0$ in Equation \eqref{main1der}, we have $\lambda_1 \lambda_2 = 0$. This is a contradiction since $\lambda_1$, $\lambda_2 \in \mathbb{T}$. Thus, $\psi^2(z) =z$ for all $z \in \mathbb{D}$. 

Now, if $\psi(z) \neq z$ for some $z \in \mathbb{D}$, choosing a function $f \in \mathcal{S}^\infty$ such that $f'(z) = 0$ and $f'(\psi(z)) = 1$, we obtain $c_1 (\lambda_1+\lambda_2) = 0$. Since $c_1 \in \mathbb{T}$, we conclude $\lambda_1 + \lambda_2 = 0$. Hence, $c_0 = \pm \lambda_1$ and $c_1 = \pm \lambda_1$. So, Equation \eqref{p} becomes
\begin{equation*} \label{gbi1}
P_i = \frac{T - \lambda_j I}{\lambda_i - \lambda_j} = \frac{T + \lambda_i I}{2 \lambda_i} = \frac{I + S_i}{2},
\end{equation*}
where $S_i = \frac{1}{\lambda_i} T$. It is clear that $S_i^2 = I$. This proves the first assertion. 

If $\psi(z) = z$ for all $z \in \mathbb{D}$, then $T(f)(z) = (c_0 - c_1)f(0) + c_1 f(z)$. Equation \eqref{p} implies that 
\begin{equation*} \label{bi1} 
P_i (f)(z) = \frac{(c_0 - c_1)f(0)+(c_1  - \lambda_j)f(z)}{\lambda_i - \lambda_j},\ i,j = 1,2,\ i \neq j.
\end{equation*}
If $c_0 = c_1 = \lambda_1$, then $T(f)(z) = \lambda_1 f(z)$. Hence, $P_1(f)(z) = f(z)$ and $P_2 = 0$. Similarly, if $c_0 = c_1 = \lambda_2$, we get $P_1 = 0$ and $P_2 (f)(z) = f(z)$. Since both $P_1$ and $P_2$ are assumed to be nonzero, these two cases are not possible.

Now, if $c_0 = \lambda_1$ and $c_1 = \lambda_2$, then $T(f)(z) = (\lambda_1 - \lambda_2)f(0) + \lambda_2 f(z)$. It follows that $P_1(f)(z) = f(0)$ and $P_2(f)(z) = f(z) - f(0)$. Consequently, for all $\mu_1, \mu_2 \in \mathbb{T}$, $f \in \mathcal{S}^\infty$ and $z \in \mathbb{D}$, we have
\begin{equation*}
\mu_1 P_1(f)(z) + \mu_2 P_2(f)(z) = \mu_1 f(0) + \mu_2 (f(z) - f(0)),
\end{equation*}
which is clearly a surjective isometry on $\mathcal{S}^\infty$. It follows that the collection $\mathcal{C}$ forms a family of bi-circular projections. Moreover, the map $S : \mathcal{S}^\infty \to \mathcal{S}^\infty$ , defined by $S(f)(z) = f(z) -2f(0)$ for all $z \in \mathbb{D}$ and $f \in \mathcal{S}^\infty$, is an isometric reflection, and $P_1 = \frac{I-S}{2}$ and $P_2 = \frac{I+S}{2}$. The case $c_0 = \lambda_2$ and $c_1 = \lambda_1$ is similar. This proves the second assertion and completes the proof. 
\end{proof}

\begin{thm} \label{formII}
Suppose the collection $\mathcal{C} = \{P_1, P_2\}$ is a family of generalized bi-circular idempotents on $\mathcal{S}^\infty$ corresponding to an isometry $T$ of the Form $(II)$. Then $\overline{c_1}T$ is an isometric reflection and one of the following holds:
\begin{enumerate}
\item Each $P_i$ is the average of the identity operator and an isometric reflection. Moreover, we have $\lambda_1 + \lambda_2 = 0$, $c_1 = \pm \lambda_1$ and $\psi^2(z) = z$ for all $z \in \mathbb{D}$.
\item $P_i (f)(z) = \frac{ c_0 \overline{f(0)} - c_1 f(0) + (c_1 - \lambda_j)f(z)}{\lambda_i - \lambda_j},\ i,j = 1,2,\ i \neq j$, where $c_1 = \lambda_1$ or $\lambda_2$. Moreover, $\psi(z) = z$ for all $z \in \mathbb{D}$.
\end{enumerate} 
\end{thm}

\begin{proof}

Let $\mathcal{C} = \{P_1, P_2\}$ be a family of generalized bi-circular idempotents on $\mathcal{S}^\infty$ corresponding to an isometry $T$ of Form $(II)$. Then $T = \lambda_1P_1 + \lambda_2 P_2$, where $\lambda_1, \lambda_2 \in \mathbb{T}$. Proposition \ref{T0=0} implies that $T(0) = 0$. Thus, $T$ takes the form	
\begin{equation*}
T(f)(z)=  c_0\overline{f(0)} + \int_{[0,z]}c_1f'(\psi(\xi))d\xi, \ \forall \ f \in \mathcal{S}^\infty \text{ and } z \in \mathbb{D}. 
\end{equation*}
Further,
\begin{equation} \label{p2}
P_i (f)(z) = \frac{T (f)(z) - \lambda_j f(z)}{\lambda_i - \lambda_j} =  \frac{1}{\lambda_i - \lambda_j} \left[c_0 \overline{f(0)} + \int_{[0,z]} c_1f'(\psi(\xi))d\xi - \lambda_j f(z)\right], 
\end{equation}
$i, j = 1,2,\ i \neq j$.

On differentiating for $i = 1$, we get
\begin{equation*}
P_1(f)'(z) = \frac{1}{\lambda_1-\lambda_2}\left[c_1f'(\psi(z)) - \lambda_2 f'(z)\right].
\end{equation*}
Taking $z=0$ in Equation \eqref{p2}, we have
\begin{equation*}
P_1(f)(0) = \frac{1}{\lambda_1-\lambda_2}\left[ c_0\overline{f(0)} - \lambda_2 f(0)\right].
\end{equation*}
Since $P_i$, $i = 1, 2$, is a  generalized bi-circular idempotent, we also have the identity
\begin{equation*} \label{identity}
TP_i (f)(z) = \lambda_i P_i (f)(z),\ i = 1,2, \ \forall \ f \in \mathcal{S}^\infty  \text{ and } z \in \mathbb{D}. 
\end{equation*}
It follows that 
\begin{align*}
c_0 \overline{ P_1(f)(0)} &+ \int_{[0,z]}c_1 P_1(f)'(\psi(\xi))d\xi \\
& = \frac{\lambda_1}{\lambda_1-\lambda_2}\left[ c_0\overline{ f(0)} +\int_{[0,z]}c_1f'(\psi(\xi))d\xi - \lambda_2 f(z)\right].
\end{align*}
Putting the values of $P_1(f)(0)$ and $P_1(f)'(z)$, we have 
\begin{align*}
\frac{c_0}{\overline{\lambda_1}-\overline{\lambda_2}} \left[\overline{ c_0\overline{f(0)}-\lambda_2 f(0) }\right] & + \frac{c_1}{\lambda_1-\lambda_2} \int_{[0,z]} \left[ c_1f'(\psi^2(\xi)) - \lambda_2 f'(\psi(\xi)) \right]  d\xi \\
& = \frac{\lambda_1}{\lambda_1-\lambda_2} \left[ c_0 \overline{f(0)}+\int_{[0,z]} c_1f'(\psi(\xi)) d\xi - \lambda_2 f(z) \right], 
\end{align*}
or
\begin{align*} 
\frac{c_0}{\overline{\lambda_1}-\overline{\lambda_2}} \left[ \overline{c_0} f(0)- \overline{\lambda_2 f(0)}\right] & + \frac{c_1}{\lambda_1-\lambda_2} \int_{[0,z]} \left[ c_1f'(\psi^2(\xi)) - \lambda_2 f'(\psi(\xi)) \right]  d\xi  \notag \\ 
& = \frac{\lambda_1}{\lambda_1-\lambda_2} \left[ c_0 \overline{f(0)}+\int_{[0,z]} c_1f'(\psi(\xi)) d\xi - \lambda_2 f(z) \right]. 
\end{align*}	
After simplification, we get 	
\begin{equation*} \label{eqbl2}
\frac{1}{\overline{\lambda_1}-\overline{\lambda_2}} f(0) + \frac{\lambda_1\lambda_2 }{\lambda_1-\lambda_2}f(z) + \frac{c_1}{\lambda_1-\lambda_2} \int_{[0,z]} \left[ c_1f'(\psi^2(\xi)) - (\lambda_1+\lambda_2) f'(\psi(\xi)) \right] d\xi = 0.
\end{equation*}	
Differentiating the above equation
\begin{equation*}\label{main2der}
c_1^2f'(\psi^2(z)) - c_1(\lambda_1+\lambda_2) f'(\psi(z))  + \lambda_1 \lambda_2 f'(z) =0.
\end{equation*}	
Choosing $f'=1$ in the above equation
\begin{equation*}
c_1^2 -c_1(\lambda_1+\lambda_2) + \lambda_1 \lambda_2 = 0  \implies  c_1 = \lambda_1,\lambda_2.
\end{equation*}	
Proceeding in the same way as we did in Theorem \ref{formI}, we conclude that $\psi^2(z) =z$ for all $z \in \mathbb{D}$. Since $\overline{c_1}T$ is an isometry of the Form $(II)$ verifying the conditions (i)-(iii) of Theorem \ref{ir2}, then $\overline{c_1}T$ is an isomeric reflection.

If $\psi (z) \neq z$ for some $z \in \mathbb{D}$, then $\lambda_1 + \lambda_2 = 0$. Moreover, $c_1 = \pm \lambda_1$. Equation \eqref{p2} implies that
\begin{equation*} \label{gbi2}
P_i = \frac{T - \lambda_j I}{\lambda_i - \lambda_j} = \frac{T + \lambda_i I}{2 \lambda_i} = \frac{I + S_i}{2},
\end{equation*}
where $S_i = \frac{1}{\lambda_i} T$. Clearly, $S_i^2 = I$. Thus, the proof of the first assertion is done. 

If $\psi(z) =z$ for all $z \in \mathbb{D}$, then $T(f)(z) = c_0 \overline{f(0)} + c_1 (f(z) - f(0))$. It follows from Equation \eqref{p2} that  
\begin{equation*}
P_i (f)(z) = \frac{ c_0 \overline{f(0)} - c_1 f(0) + (c_1 - \lambda_j)f(z)}{\lambda_i - \lambda_j},\ i,j = 1,2,\ i \neq j,
\end{equation*}
where  $c_1 = \lambda_1$ or $\lambda_2$. The proof of the second assertion is complete. 
\end{proof}

Theorems \ref{formI} and \ref{formII} characterize generalized bi-circular idempotents associated with isometries of Forms $(I)$ and $(II)$. We now consider isometries of Forms $(III)$ and $(IV)$ and investigate the structure of the corresponding generalized bi-circular idempotents in the following two theorems.

\begin{thm} \label{formIII}
Let $\mathcal{C} = \{P_1, P_2\}$ be a family of generalized bi-circular idempotents on $\mathcal{S}^\infty$ corresponding to an isometry $T$ of Form $(III)$. Then $c_0 = \lambda_1, \lambda_2$, $\psi(\overline{\psi(\overline{z})}) = z$ for all $z \in \mathbb{D}$, $\overline{c_0}T$ is an isometric reflection and 
\begin{equation*}
P_i (f)(z) =  \frac{1}{\lambda_i - \lambda_j} \left[ c_0f(0) + \int_{[0,z]} c_1 \overline{ f'(\psi(\overline{\xi}))}d\xi - \lambda_j f(z)\right],\ i, j = 1,2,\ i \neq j. 
\end{equation*}
\end{thm}

\begin{proof}
Let $\mathcal{C} = \{P_1, P_2\}$ be a family of generalized bi-circular idempotents on $\mathcal{S}^\infty$ corresponding to an isometry $T$ of Form $(III)$. Then $T = \lambda_1P_1 + \lambda_2 P_2$, where $\lambda_1, \lambda_2 \in \mathbb{T}$. Proposition \ref{T0=0} implies that $T(0) = 0$. Thus, $T$ takes the form			
\begin{equation*}
T(f)(z)= c_0f(0) + \int_{[0,z]} c_1 \overline{f'(\psi(\overline{\xi}))}d\xi,\ \forall \ f \in \mathcal{S}^\infty \text{ and } z \in \mathbb{D}. 
\end{equation*}
We also have,
\begin{equation} \label{p3}
P_i (f)(z) = \frac{T(f)(z) - \lambda_j f(z)}{\lambda_i - \lambda_j}  =  \frac{1}{\lambda_i - \lambda_j} \left[c_0f(0)+\int_{[0,z]} c_1 \overline{f'(\psi(\overline{\xi}))}d\xi - \lambda_j f(z)\right], 
\end{equation}
$i, j = 1,2,\ i \neq j$.

Differentiating the above equation for $i = 1$ (recall that the function $z \mapsto \overline{f'(\psi(\overline{\xi}))}$ is holomorphic on $\mathbb{D}$),  we get
\begin{equation*}
P_1(f)'(z) = \frac{1}{\lambda_1-\lambda_2} \left[ c_1 \overline{f'(\psi(\overline{z}))} - \lambda_2 f'(z)\right].
\end{equation*}
Taking $z=0$ in Equation \eqref{p3}, we obtain
\begin{equation*}
P_1 (f)(0) = \frac{1}{\lambda_1-\lambda_2} \left[c_0f(0)-\lambda_2 f(0)\right].
\end{equation*}
Using all this information in the identity $TP_1(f)(z) = \lambda_1 P_1(f)(z)$, 
\begin{align*}
c_0 P_1(f)(0) + \int_{[0,z]} c_1 \overline{P_1(f)'(\psi(\overline{\xi}))}d\xi
= \frac{\lambda_1}{\lambda_1 - \lambda_2} \left[ c_0f(0) + \int_{[0,z]}  c_1 \overline{f'(\psi(\overline{\xi}))}d\xi - \lambda_2 f(z)\right],
\end{align*}
or
\begin{align*}
\frac{c_0}{\lambda_1-\lambda_2}[c_0f(0) - \lambda_2 f(0)] + \frac{c_1}{\overline{\lambda_1}-\overline{\lambda_2}} \int_{[0,z]}  \left[\overline{ c_1 \overline{f'(\psi(\overline{\psi(\overline{\xi})}))} - \lambda_2 f'(\psi(\overline{\xi}))}\right]  d\xi \\
= \frac{\lambda_1}{\lambda_1-\lambda_2} \left[c_0f(0)  + \int_{[0,z]} c_1 \overline{f'(\psi(\overline{\xi}))}d\xi - \lambda_2 f(z)\right], 
\end{align*}
or
\begin{equation}\label{main3}
\frac{1}{\lambda_1-\lambda_2}[(c_0^2 - c_0\lambda_1 -c_0\lambda_2) f(0) + \lambda_1\lambda_2 f(z)] + \frac{1}{\overline{\lambda_1}-\overline{\lambda_2}} \int_{[0,z]} f'(\psi(\overline{\psi(\overline{\xi})})) d\xi = 0.
\end{equation}
Choosing $f = 1$, we get 
\begin{equation*}	
c_0^2-c_0 \lambda_1 -c_0 \lambda_2 + \lambda_1\lambda_2 =0  \implies  c_0 =\lambda_1,\lambda_2.
\end{equation*}
Differentiating Equation \eqref{main3}, we obtain 
\begin{equation*} \label{main3der}
\frac{\lambda_1\lambda_2}{\lambda_1-\lambda_2}f'(z) + \frac{1}{\overline{\lambda_1} -\overline{\lambda_2}} f'(\psi(\overline{\psi(\overline{z})})) = 0. 
\end{equation*}	
This implies that $\psi(\overline{\psi(\overline{z})}) = z$ for all $z \in \mathbb{D}$. Since $\overline{c_0}T$ is an isometry of the Form $(III)$ verifying the conditions (i)-(iii) of Theorem \ref{ir3}, then $\overline{c_0}T$ is an isomeric reflection. This completes the proof.
\end{proof}	
\begin{thm} \label{formIV}
If $\mathcal{C} = \{P_1, P_2\}$ is a family of generalized bi-circular idempotents on $\mathcal{S}^\infty$ corresponding to an isometry $T$ of the Form $(IV)$, then $\psi(\overline{\psi(\overline{z})}) = z$ for all $z \in \mathbb{D}$, $T$ is an isometric reflection and 
\begin{equation*} 
P_i (f)(z) =  \frac{1}{\lambda_i-\lambda_j}\left[ c_0 \overline{f(0)} + \int_{[0,z]} c_1 \overline{f'(\psi(\overline{\xi}))}d\xi - \lambda_j f(z)\right],\ i, j = 1,2,\ i \neq j. 
\end{equation*}
\end{thm}
\begin{proof}
Let $\mathcal{C} = \{P_1, P_2\}$ be a family of generalized bi-circular idempotents on $\mathcal{S}^\infty$ corresponding to an isometry $T$ of Form $(IV)$. Then $T = \lambda_1P_1 + \lambda_2 P_2$, where $\lambda_1, \lambda_2 \in \mathbb{T}$. Proposition \ref{T0=0} implies that $T(0) = 0$. Thus, $T$ takes the form		
\begin{equation*}
T(f)(z)= c_0 \overline{f(0)} + \int_{[0,z]} c_1 \overline{f'(\psi(\overline{\xi}))}d\xi,\ \forall \ f \in \mathcal{S}^\infty,  z \in \mathbb{D}. 
\end{equation*}
Further,
\begin{equation*} \label{p4}
P_i (f)(z) = \frac{T (f)(z) - \lambda_j f(z)}{\lambda_i - \lambda_j} =  \frac{1}{\lambda_i-\lambda_j}\left[ c_0 \overline{f(0)} + \int_{[0,z]} c_1 \overline{f'(\psi(\overline{\xi}))}d\xi - \lambda_j f(z)\right],\ 
\end{equation*}
$i, j = 1,2,\ i \neq j$.
This implies that
\begin{equation*}
P_1(f)'(z) =\frac{1}{\lambda_1-\lambda_2} \left[ c_1 \overline{f'(\psi(\overline{z}))} - \lambda_2 f'(z)\right],
\end{equation*}
and
\begin{equation*}
P_1(f)(0) = \frac{1}{\lambda_1-\lambda_2}\left[ c_0 \overline{f(0)} - \lambda_2 f(0)\right].
\end{equation*}
$T P_1(f)(z) = \lambda_1 P_1(f)(z) \implies$
\begin{align*}
c_0 \overline{ P_1(f)(0)} &+ \int_{[0,z]} c_1\overline{P_1(f)'(\psi(\overline{\xi}))}d\xi \\ 
& =  \frac{\lambda_1}{\lambda_1-\lambda_2}\left[ c_0 \overline{f(0)} + \int_{[0,z]} c_1 \overline{f'(\psi(\overline{\xi}))}d\xi - \lambda_2 f(z)\right].
\end{align*}
It follows that 
\begin{align*}
\frac{c_0}{\overline{\lambda_1} - \overline{\lambda_2}} \left[\overline{c_0 \overline{f(0)} - \lambda_2 f(0)}\right]  
& + \frac{c_1}{\overline{\lambda_1} - \overline{\lambda_2}} \int_{[0,z]} \left[ \overline{ c_1 \overline{f'(\psi(\overline{\psi(\overline{\xi})}))} - \lambda_2 f'(\psi(\overline{\xi}))}\right] d\xi \\
& = \frac{\lambda_1}{\lambda_1-\lambda_2} \left[ c_0\overline{f(0)} +  \int_{[0,z]} c_1 \overline{f'(\psi(\overline{\xi}))}d\xi - \lambda_2 f(z)\right],
\end{align*}
or  
\begin{align*} \label{eqbl4}
\frac{1}{\overline{\lambda_1} - \overline{\lambda_2}}f(0)+ \frac{\lambda_1\lambda_2}{\lambda_1-\lambda_2} f(z)  + \frac{1}{\overline{\lambda_1} - \overline{\lambda_2}}  \int_{[0,z]} f'(\psi(\overline{\psi(\overline{\xi})})) d\xi = 0.
\end{align*}	
Now, differentiating the above equation, we get 
\begin{equation*} \label{main4der}
\frac{ \lambda_1\lambda_2}{\lambda_1-\lambda_2} f'(z) + \frac{1}{\overline{\lambda_1} - \overline{\lambda_2}} f'(\psi(\overline{\psi(\overline{z})}))= 0.
\end{equation*}	
This implies that $\psi(\overline{\psi(\overline{z})}) = z$ for all $z \in \mathbb{D}$. Since $T$ is an isometry of the Form $(IV)$ verifying the conditions of Theorem \ref{ir4}, then $T$ is an isometric reflection. This completes the proof. 	
\end{proof}

Based on previous results, we point out that:

\begin{rem} 
Although Corollary \ref{gbireal} says that any $GBI$ is a real linear map, it is clear from the structures of $GBI$ obtained in Theorems \ref{formII}, \ref{formIII}, and \ref{formIV} that they are nonlinear maps.
\end{rem}

\begin{cor}
If $\mathcal{C} = \{P_1, P_2\}$ is a family of generalized bi-circular idempotents on $\mathcal{S}^\infty$ corresponding to an isometry $T$, each $P_i$ is a linear combination of the identity operator and an isometric reflection. Moreover, one of the following holds:
\begin{enumerate}
\item There is $\alpha \in \mathbb{T}$ such that $\alpha T$ is an isometric reflection.
\item  $T$ is of the Form $(I)$ and the map $S : \mathcal{S}^\infty \to \mathcal{S}^\infty$, given by 
\begin{equation*}
S(f)(z) = \overline{c_1}\big(T(f)(z) - (c_0 + c_1)f(0) \big), \  \forall \ f  \in \mathcal{S}^\infty, z \in \mathbb{D},
\end{equation*}
is an isometric reflection.
\end{enumerate}
\end{cor}

\textbf{Conflict of Interest:} The authors declare that they have no conflict of interest.

\end{document}